\documentclass[12pt, a4paper, leqno]{article}
\usepackage[margin=2.2cm]{geometry}
\usepackage[english]{babel}

\usepackage[runin]{abstract}
\usepackage{titling}

\usepackage{amsmath}
\usepackage{amsthm}
\usepackage{amsfonts}
\usepackage{amssymb}
\usepackage{mathtools}
\usepackage{clrscode}
\usepackage{enumitem}
\usepackage{mdwlist}

\abslabeldelim{.}
\newcommand{\keywordsname}{Key words}
\makeatletter
\newcommand{\keywords}[1]{%
\def\thekeywords{#1}%
\begin{@bstr@ctlist}
\hspace*{\abstitleskip}{\abstractnamefont\keywordsname\@bslabeldelim}\abstracttextfont\
#1%
\par\end{@bstr@ctlist}
}
\makeatother
\newcommand{\subjclassname}{Mathematics subject classification}
\makeatletter
\newcommand{\subjclass}[2][2010]{%
\begin{@bstr@ctlist}
\hspace*{\abstitleskip}{\abstractnamefont\subjclassname\ (#1)\@bslabeldelim}\abstracttextfont\
#2%
\par\end{@bstr@ctlist}
}
\makeatother
\makeatletter
\def\and{%				%begin{tabular}
	\end{tabular}%
	and%
	\begin{tabular}[t]{c}}%
\makeatother
\makeatletter
\def\thanks#1{%\footnotemark
\protected@xdef\@thanks{\@thanks
\protect\footnotetext[\the\c@footnote]{#1}}%
}
\makeatother
\makeatletter
\let\addresses\@empty      %\let\thankses\@empty
\newcommand{\address}[2][]{\g@addto@macro\addresses{\address{#1}{#2}}}
\newcommand{\curraddr}[2][]{\g@addto@macro\addresses{\curraddr{#1}{#2}}}
\newcommand{\email}[2][]{\g@addto@macro\addresses{\email{#1}{#2}}}
\newcommand{\urladdr}[2][]{\g@addto@macro\addresses{\urladdr{#1}{#2}}}
\def\enddoc@text{%\ifx\@empty\@translators \else\@settranslators\fi
  \ifx\@empty\addresses \else\@setaddresses\fi}
\AtEndDocument{\enddoc@text}
\def\emailaddrname{E-mail address}
\def\@setaddresses{\par
  \nobreak \begingroup
  \interlinepenalty\@M
  \def\address##1##2{\begingroup%
    \par\addvspace\bigskipamount%\indent
    \@ifnotempty{##1}{(\ignorespaces##1\unskip) }%
    {\noindent\ignorespaces##2}\par\endgroup}%
  \def\email##1##2{\begingroup
    \@ifnotempty{##2}{\nobreak\noindent\emailaddrname
      \@ifnotempty{##1}{, \ignorespaces##1\unskip}\/:\space
      \ttfamily##2\par}\endgroup}%
  \addresses
  \endgroup
}

\makeatother
\makeatletter
\def\cstar#1{\expandafter\@cstar\csname c@#1\endcsname}
\def\@cstar#1{\ifcase#1\or $\ast$\or $\ast\ast$\or $\ast\ast\ast$\fi}
\AddEnumerateCounter{\cstar}{\@cstar}{$\ast\ast\ast$}
\makeatother

\newlist{conditions}{enumerate}{1}
\setlist[conditions]{label=\normalfont(\alph*),ref=\normalfont\alph*}

\mathchardef\mhyphen="2D
\newcommand{\PB}{\mathbb{P}}
\newcommand{\R}{\mathbb{R}}

\newcommand{\C}{\mathcal{C}}

\newcommand{\dom}{\func{dom}}
\newcommand{\pole}{\func{pole}}

\newtheorem{theorem}{Theorem}[section]
\newtheorem{corollary}[theorem]{Corollary}
\newtheorem{lemma}[theorem]{Lemma}
\theoremstyle{definition}
\newtheorem{convention}[theorem]{Convention}
\newtheorem{definition}[theorem]{Definition}
\newtheorem{example}[theorem]{Example}
\newtheorem{notation}[theorem]{Notation}
\newtheorem{remark}[theorem]{Remark}

\DeclarePairedDelimiter\abs{\lvert}{\rvert}%
\DeclarePairedDelimiter\norm{\lVert}{\rVert}%

\makeatletter
\let\oldabs\abs
\def\abs{\@ifstar{\oldabs}{\oldabs*}}
\let\oldnorm\norm
\def\norm{\@ifstar{\oldnorm}{\oldnorm*}}
\makeatother
\title{\bf Linear equations on real algebraic surfaces}
\date{}
\author{Wojciech Kucharz\thanks{The first author was partially supported
by the National Science Centre (Poland), under grant number
2014/15/B/ST1/00046. He also acknowledges with gratitude support and
hospitality of the Max--Planck--Institut f\"ur Mathematik in Bonn.} \and
Krzysztof Kurdyka\thanks{The second author was partially supported by
ANR project STAAVF (France).}}

\address{Wojciech Kucharz\\Institute of Mathematics\\Faculty of Mathematics and Computer
Science\\Jagiellonian University\\\L{}ojasiewicza 6\\30-348
Krak\'ow\\Poland}
\email{Wojciech.Kucharz@im.uj.edu.pl}

\address{Krzysztof Kurdyka\\Laboratoire de Math\'ematiques\\UMR 5175 du
CNRS\\Universit\'e  Savoie Mont Blanc \\Campus Scientifique\\73 376 Le
Bourget-du-Lac Cedex\\France}
\email{kurdyka@univ-savoie.fr}

\usepackage[pdftex, pdfauthor={\theauthor}, pdftitle={\thetitle}]{hyperref}

\begin{document}
\maketitle
\thispagestyle{empty}

\begin{abstract}
We prove that if a linear equation, whose coefficients are continuous
rational functions on a nonsingular real algebraic surface, has a
continuous solution, then it also has a continuous rational solution.
This is known to fail in higher dimensions.
\end{abstract}
%\cleardobulepage

\hypersetup{pdfkeywords={\thekeywords}}

\section{Introduction}\label{sec:1}

Fefferman and Koll\'ar \cite{bib4} study the following problem. Given
continuous functions $f_1, \ldots, f_r$ on $\R^n$, which continuous
functions $\varphi$ can be written in the form
\begin{equation*}
\varphi = \varphi_1 f_1 + \cdots + \varphi_r f_r,
\end{equation*}
where the $\varphi_i$ are continuous functions on $\R^n$? Moreover, if
$\varphi$ and the $f_i$ have some regularity properties, can one choose
the $\varphi_i$ to have the same (or weaker) regularity properties? In
other words, the questions are about solutions of linear equations of
the form
\begin{equation*}
f_1 y_1 + \cdots + f_r y_r = \varphi.
\end{equation*}

The problem is hard even if $\varphi$ and the $f_i$ are polynomial
functions. In \cite{bib4}, two different ways to solve the problem are
presented: the Glaeser--Michael method and the algebraic geometry
approach. Each of them consists of a rather complex procedure and it
does not seem possible to give a concise answer in general.

In this note we settle the problem in a simple manner, assuming that
$n=2$ and the $f_i$ are continuous rational functions. Actually, our
results are more general and settle the corresponding problem for
functions defined on any nonsingular real algebraic surface.

A~complex version of the problem under consideration was studied by
Brenner \cite{bib3}, Epstein and Hochster \cite{new}, and Koll\'ar
\cite{bib8}.

\begin{convention}\label{conv-1-1}
By a function we always mean a real-valued function.
\end{convention}

\begin{notation}\label{not-1-2}
If $f_1, \ldots, f_r$ are functions defined on some set $S$, then
\begin{equation*}
Z(f_1, \ldots, f_r) \coloneqq \{ x \in S \mid f_1(x)=0, \ldots, f_r(x)=0
\}.
\end{equation*}
\end{notation}

We now recall the \emph{pointwise test} (or PT for short) introduced in
\cite[p.~235]{bib4}.

\begin{definition}\label{def-1-3}
Let $\Omega$ be a metric space and let $f_1, \ldots, f_r$ be continuous
functions on $\Omega$. We say that a continuous function $\varphi$ on
$\Omega$ satisfies the PT for the functions  $f_i$ if for every point $p \in
\Omega$, the following two equivalent conditions hold:
\begin{conditions}
\item\label{def-1-3-a} The function $\varphi$ can be written as
\begin{equation*}
\varphi = \psi_1^{(p)} f_1 + \cdots + \psi_r^{(p)} f_r,
\end{equation*}
where the $\psi_i^{(p)}$ are functions on $\Omega$ that are continuous at
$p$.

\item\label{def-1-3-b} The function $\varphi$ can be written as
\begin{equation*}
\varphi = \varphi^{(p)} + c_1^{(p)} f_1 + \cdots + c_r^{(p)} f_r,
\end{equation*}
where $c_i^{(p)} \in \R$ and the functions $A_i^{(p)}$ defined by
\begin{equation*}
A_i^{(p)} = \frac{\varphi^{(p)}f_i}{f_1^2 + \cdots + f_r^2} \quad \textrm{on }
\Omega \setminus Z^{(p)} \quad \textrm{and} \quad A_i^{(p)}=0 \quad \textrm{on }
Z^{(p)},
\end{equation*}
with $Z^{(p)} \coloneqq Z(f_1, \ldots, f_r) \cup \{p\}$, are continuous at
$p$.
\end{conditions}
\end{definition}

Note that conditions (\ref{def-1-3-a}) and (\ref{def-1-3-b}) are indeed
equivalent. If (\ref{def-1-3-a}) holds, then so does (\ref{def-1-3-b})
with
\begin{equation*}
\varphi^{(p)} = (\psi_1^{(p)} - \psi_1^{(p)}(p)) f_1 + \cdots +
(\psi_r^{(p)} - \psi_r^{(p)}(p)) f_r \quad \textrm{and} \quad c_i^{(p)}
= \psi_i^{(p)}(p).
\end{equation*}
Conversely, (\ref{def-1-3-b}) implies (\ref{def-1-3-a}) with
$\psi_i^{(p)} = c_i^{(p)} + A_i^{(p)}$.

Clearly, the PT is a basic necessary condition for existence of
continuous functions $\varphi_1, \ldots, \varphi_r$ on $\Omega$
satisfying $\varphi = \varphi_1 f_1 + \cdots + \varphi_r f_r$.

For background on real algebraic geometry the reader may consult
\cite{bib2}. By a \emph{real algebraic variety} we mean a locally ringed
space isomorphic to an algebraic subset of $\R^n$, for some~$n$, endowed
with the Zariski topology and the sheaf of regular functions (such an
object is called an affine real algebraic variety in \cite{bib2}). Recall
that any quasi-projective real algebraic variety is a real algebraic
variety in the sense just defined, cf.
\cite[Prop.~3.2.10, Thm.~3.4.4]{bib2}. Each real algebraic variety
carries also the Euclidean topology, which is determined by the usual
metric on $\R$. Unless explicitly stated otherwise, all topological
notions relating to real algebraic varieties will refer to the Euclidean
topology.

We say that a function $f$ defined on a real algebraic variety $X$ is
\emph{continuous rational} if it is continuous on $X$ and regular on
some Zariski open dense subset of $X$. We denote by $P(f)$ the smallest
Zariski closed subset of $X$ such that $f$ is regular on $X \setminus
P(f)$. The continuous rational functions form a subring of the ring of
all continuous functions on $X$. Any regular function on $X$ is
continuous rational. The converse does not hold in general, even if $X$
is nonsingular.

\begin{example}\label{ex-1-4}
The function $f$ on $\R^2$, defined by
\begin{equation*}
f(x,y) = \frac{x^3}{x^2+y^2} \quad \textrm{for} \ (x,y) \neq (0,0) \quad
\textrm{and} \quad f(0,0)=0,
\end{equation*}
is continuous rational but it is not regular; in fact, $P(f) = \{ (0,0)
\}$.
\end{example}

Recently, continuous rational functions have attracted a lot of
attention, cf. \cite{bib1, bib5, bib6, bib9, bib10, bib11, bib12, bib13,
bib14, bib15, bib16, bib17, bib18, bib19}. On nonsingular varieties they
coincide with regulous functions introduced by Fichou, Huisman, Mangolte
and Monnier \cite{bib5}.

Our first result, to be proved in Section~\ref{sec:2}, is the following.

\begin{theorem}\label{th-1-5}
Let $X$ be a nonsingular real algebraic surface and let $f_1, \ldots,
f_r$ be continuous rational functions on $X$. For a continuous function
$\varphi$ on $X$, the following conditions are equivalent:
\begin{conditions}
\item\label{th-1-5-a} The function $\varphi$ can be written in the form
\begin{equation*}
\varphi = \varphi_1 f_1 + \cdots + \varphi_r f_r,
\end{equation*}
where the $\varphi_i$ are continuous functions on $X$.

\item\label{th-1-5-b} The function $\varphi$ satisfies the PT for the $f_i$.
\end{conditions}
\end{theorem}

An example of Hochster \cite[p.~236]{bib4}, which involves simple
polynomial functions on $\R^3$, shows that Theorem~\ref{th-1-5} cannot
be extended to varieties of higher dimension.

In Section~\ref{sec:3} we prove the following.

\begin{theorem}\label{th-1-6}
Let $X$ be a nonsingular real algebraic surface and let $f_1, \ldots,
f_r$ be continuous rational functions on $X$. For a continuous rational
function $\varphi$ on $X$, the following conditions are equivalent:
\begin{conditions}
\item\label{th-1-6-a} The function $\varphi$ can be written in the form
\begin{equation*}
\varphi = \varphi_1 f_1 + \cdots + \varphi_r f_r,
\end{equation*}
where the $\varphi_i$ are continuous rational functions on $X$.

\item\label{th-1-6-b} The function $\varphi$ satisfies the PT for the
$f_i$.

\end{conditions}
Furthermore, if {\normalfont(\ref{th-1-6-b})} holds, then the $\varphi_i$ in
{\normalfont(\ref{th-1-6-a})} can be chosen so that $P(\varphi_i)$ is a finite set
contained in $Z(f_1, \ldots, f_r) \cup P(f_1) \cup \ldots \cup P(f_r)
\cup P(\varphi)$.
\end{theorem}

As a straightforward consequence we get

\begin{corollary}\label{cor-1-7}
Let $X$ be a nonsingular real algebraic surface and let $f_1, \ldots,
f_r$ be continuous rational functions on $X$. For a continuous rational
function $\varphi$ on $X$, the following conditions are equivalent:
\begin{conditions}
\item\label{cor-1-7-a} The function $\varphi$ can be written in the form
\begin{equation*}
\varphi = \varphi_1 f_1 + \cdots + \varphi_r f_r,
\end{equation*}
where the $\varphi_i$ are continuous rational functions on $X$.

\item\label{cor-1-7-b} The function $\varphi$ can be written in the form
\begin{equation*}
\varphi = \psi_1 f_1 + \cdots + \psi_r f_r,
\end{equation*}
where the $\psi_i$ are continuous functions on $X$.
\end{conditions}
\end{corollary}

Corollary~\ref{cor-1-7} cannot be extended to varieties of higher
dimension. A relevant example, involving polynomial functions on $\R^3$,
is provided by Koll\'ar and Nowak \cite[Example~6]{bib9}. Furthermore,
the argument used in \cite[Example~6]{bib9} shows that
Corollary~\ref{cor-1-7} does not hold for the singular real algebraic
surface $S \subset \R^3$ that appears there.

We conclude this section with an example.

\begin{example}\label{ex-1-8}
Consider the functions $f_1(x,y) = x^3$, $f_2(x,y) = y^3$, $\varphi(x,y)
= x^2 y^2$ on $\R^2$. We have
\begin{equation*}
\varphi = \varphi_1 f_1 + \varphi_2 f_2,
\end{equation*}
where $\varphi_1$, $\varphi_2$ are continuous rational functions on
$\R^2$ defined by
\begin{align*}
\varphi_1(x,y) &= \frac{x^5 y^2}{x^6 + y^6} \quad \textrm{for} \ (x,y)
\neq (0,0), \quad \varphi_1(0,0)=0,\\
\varphi_2(x,y) &= \frac{x^2 y^5}{x^6 + y^6} \quad \textrm{for} \ (x,y)
\neq (0,0), \quad \varphi_2(0,0)=0.
\end{align*}
However, $\varphi$ cannot be written as a linear combination of $f_1$
and $f_2$ with coefficients that are regular (or $\C^{\infty}$) functions
on $\R^2$, as can be seen by comparing the Taylor's expansions at
$(0,0)$.
\end{example}

\section{Continuous solutions}\label{sec:2}

We begin with some preliminary results.

\begin{lemma}\label{lem-2-1}
Let $\Omega$ be a metric space and let $f_1, \ldots, f_r, \varphi$ be
continuous functions on $\Omega$ such that the set $Z(f_1, \ldots, f_r)$
is nowhere dense in $\Omega$ and $\varphi$ satisfies the PT for the
$f_i$. Assume that $f_i = gg_i$, where $g$ and the $g_i$ are continuous
functions on $\Omega$. Then there exists a unique continuous function
$\psi$ on $\Omega$ such that $\varphi = g\psi$. Furthermore, $\psi$
satisfies the PT for the $g_i$.
\end{lemma}

\begin{proof}
Note that the set $Z(g)$ is nowhere dense in $\Omega$. To prove
existence of $\psi$ (uniqueness is then automatic) it suffices to show
that for every point $p \in \Omega$ the limit
\begin{equation*}
\lim_{x\to p} \frac{\varphi(x)}{g(x)},\quad \textrm{where}\ x\in \Omega
\setminus Z(g),
\end{equation*}
exists. This readily follows since $\varphi$ can be written as
\begin{equation*}
\varphi = \psi_1^{(p)} f_1 + \cdots + \psi_r^{(p)} f_r = g(\psi^{(p)}g_1
+ \cdots + \psi_r^{(p)}g_r),
\end{equation*}
where the $\psi_i^{(p)}$ are functions on $\Omega$ that are continuous
at $p$.

It remains to prove that $\psi$ satisfies the PT for the $g_i$. We set
$Z^{(p)} \coloneqq Z(f_1, \ldots, f_r) \cup \{p\}$ and write $\varphi$
in the form
\begin{equation*}
\varphi = \varphi^{(p)} + c_1^{(p)}f_1 + \cdots + c_r^{(p)}f_r =
\varphi^{(p)} + g(c_1^{(p)}g_1 + \cdots + c_r^{(p)}g_r),
\end{equation*}
where $c_i^{(p)} \in \R$ and the functions $A_i^{(p)}$, defined by
\begin{equation*}
A_i^{(p)} = \frac{\varphi^{(p)}f_i}{f_1^2 + \cdots + f_r^2} \quad
\textrm{on}\ \Omega\setminus Z^{(p)}\quad \textrm{and}\quad A_i^{(p)}=0 \quad
\textrm{on}\ Z^{(p)},
\end{equation*}
are continuous at $p$. Defining $\psi^{(p)}$ by
\begin{equation*}
\psi = \psi^{(p)} + c_1^{(p)}g_1 + \cdots + c_r^{(p)}g_r,
\end{equation*}
we get $\varphi^{(p)} = g\psi^{(p)}$. Consequently,
\begin{equation*}
\frac{\varphi^{(p)}f_i}{f_1^2 + \cdots + f_r^2} =
\frac{\psi^{(p)}g_i}{g_1^2 + \cdots + g_r^2} \quad \textrm{on} \ \Omega
\setminus Z(f_1, \ldots, f_r).
\end{equation*}
Since the set $Z(f_1, \ldots f_r)$ is nowhere dense in $\Omega$, it
follows that $\psi$ satisfies the PT for the~$g_i$.
\end{proof}

\begin{lemma}\label{lem-2-2}
Let $X$ be an irreducible nonsingular real algebraic variety and let
$f_1, \ldots, f_r$ be continuous rational functions on $X$, not all
identically equal to $0$. Then the Zariski closure of $Z(f_1, \ldots,
f_r)$ is Zariski nowhere dense in $X$. In particular, $Z(f_1, \ldots,
f_r)$ is Euclidean nowhere dense in $X$.
\end{lemma}

\begin{proof}
Setting $f = f_1^2 + \cdots + f_r^2$, we get $Z(f)=Z(f_1, \ldots, f_r)$.
The function $f$ is continuous rational and satisfies
\begin{equation*}
Z(f)\subset Z(f|_{X \setminus P(f)}) \cup P(f),
\end{equation*}
which implies both assertions.
\end{proof}

\begin{lemma}\label{lem-2-3}
Let $X$ be an irreducible nonsingular real algebraic surface and let
$f_1, \ldots, f_r$ be continuous rational functions on $X$. Then, for
every point $p \in X$, there exists a Zariski open neighborhood
$X^{(p)} \subset X$ of $p$ and there exist regular functions $g_1,
\ldots, g_r, g, h$ on $X^{(p)}$ such that $Z(h) \neq X^{(p)}$, $Z(g_1,
\ldots, g_r) \subset \{p\}$, and $hf_i = gg_i$ on $X^{(p)}$ for
$i=1,\ldots, r$.
\end{lemma}

\begin{proof}
We can find regular functions $\lambda_1, \ldots, \lambda_r, \mu$ on $X$
such that $Z(\mu) \neq X$ and $f_i = \lambda_i / \mu$ on $X \setminus
Z(\mu)$ for $i=1,\ldots,r$. Since $X$ is nonsingular, the local ring of regular functions
 at each point $p\in X$ is a unique factorization domain.
Consequently, there exists a Zariski open neighborhood $X^{(p)} \subset
X$ of $p$ and there exist regular functions $g_1, \ldots, g_r, g$ on
$X^{(p)}$ such that $\lambda_i = gg_i$ on $X^{(p)}$ and $Z(g_1, \ldots,
g_r) \subset \{p\}$. To complete the proof it suffices to set $h
\coloneqq \mu|_{X^{(p)}}$.
\end{proof}

\begin{lemma}\label{lem-2-4}
Let $X$ be an irreducible nonsingular real algebraic surface and let
$f_1, \ldots, f_r$ be continuous rational functions on $X$, not all
identically equal to $0$. Let $\varphi$ be a continuous function on $X$
that satisfies the PT for the $f_i$. Then, for every point $p \in X$,
there exists a~Zariski open neighborhood $X^{(p)} \subset X$ of $p$ and
there exist continuous functions $\alpha_1^{(p)}, \ldots, \alpha_r^{(p)}$
on $X^{(p)}$ and real numbers $c_1^{(p)}, \ldots, c_r^{(p)}$ such that
\begin{gather*}
\varphi = \alpha_1^{(p)}f_1 + \cdots + \alpha_r^{(p)}f_r \quad \textrm{on} \
X^{(p)},\quad \textrm{and}\\
\alpha_i^{(p)} = c_i^{(p)} + \frac{(\varphi - (c_1^{(p)}f_1 + \cdots +
c_r^{(p)}f_r))f_i}{f_1^2 + \ldots + f_r^2} \quad \textrm{on} \ X^{(p)} \setminus
Z(f_1, \ldots, f_r)
\end{gather*}
for $i=1,\ldots,r$.
\end{lemma}

\begin{proof}
By Lemma~\ref{lem-2-3}, we can find a Zariski open neighborhood $X^{(p)}
\subset X$ of $p$ and regular functions $g_1, \ldots, g_r, g, h$ on
$X^{(p)}$ such that
\begin{equation}\label{eq-2-4-1}
Z(g_1,\ldots,g_r) \subset \{p\},
\end{equation}
\begin{equation}\label{eq-2-4-2}
hf_i = gg_i \quad \textrm{on} \ X^{(p)} \quad \textrm{for}\ i=1, \ldots, r,
\end{equation}
and $Z(h) \neq X^{(p)}$. Since $\varphi$ satisfies the PT for the $f_i$,
it follows that $h\varphi|_{X^{(p)}}$ satisfies the PT for the
$hf_i|_{X^{(p)}} = gg_i$. According to Lemma~\ref{lem-2-2}, the set
\begin{equation*}
Z(hf_1|_{X^{(p)}}, \ldots, hf_r|_{X^{(p)}}) = Z(gg_1, \ldots, gg_r)
\end{equation*}
is nowhere dense in $X^{(p)}$. Hence, in view of Lemma~\ref{lem-2-1},
there exists a unique continuous function $\psi$ on $X^{(p)}$ such that
\begin{equation}\label{eq-2-4-3}
h\varphi|_{X^{(p)}} = g\psi.
\end{equation}
Furthermore, $\psi$ satisfies the PT for the $g_i$. Consequently, taking
\eqref{eq-2-4-1} into account, we can write $\psi$ in the form
\begin{equation}\label{eq-2-4-4}
\psi = \psi^{(p)} + c_1^{(p)}g_1 + \cdots + c_r^{(p)}g_r,
\end{equation}
where $c_i^{(p)} \in \R$ and the functions $B_i^{(p)}$ on $X^{(p)}$,
defined by
\begin{equation}\label{eq-2-4-5}
B_i^{(p)} = \frac{\psi^{(p)}g_i}{g_1^2 + \cdots + g_r^2} \quad \textrm{on}\
X^{(p)} \setminus \{p\}\quad \textrm{and} \quad B_i^{(p)}(p) = 0,
\end{equation}
are continuous at $p$. It follows that the $B_i^{(p)}$ are continuous on
$X^{(p)}$.

Defining $\varphi^{(p)}$ by
\begin{equation}\label{eq-2-4-6}
\varphi = \varphi^{(p)} + c_1^{(p)}f_1 + \cdots + c_r^{(p)}f_r
\end{equation}
and making use of \eqref{eq-2-4-2}--\eqref{eq-2-4-6}, we get
\begin{equation*}
B_i^{(p)} = \frac{\varphi^{(p)}f_i}{f_1^2 + \cdots + f_r^2}\quad \textrm{on}
\ X^{(p)} \setminus (\{p\}\cup Z(g)).
\end{equation*}
By continuity,
\begin{equation}\label{eq-2-4-7}
B_i^{(p)} = \frac{\varphi^{(p)}f_i}{f_1^2 + \cdots + f_r^2}\quad \textrm{on}
\ X^{(p)} \setminus Z(f_1, \ldots, f_r).
\end{equation}
The functions $\alpha_i^{(p)} \coloneqq c_i^{(p)} + B_i^{(p)}$ are
continuous on $X^{(p)}$ and in view of \eqref{eq-2-4-6}, \eqref{eq-2-4-7}
they satisfy
\begin{equation*}
\varphi = \alpha_1^{(p)}f_1 + \cdots + \alpha_r^{(p)}f_r \quad \textrm{on} \
X^{(p)} \setminus Z(f_1, \ldots, f_r).
\end{equation*}
By continuity, the last equality holds on $X^{(p)}$. The proof is
complete.
\end{proof}

\begin{proof}[Proof of Theorem~\ref{th-1-5}]
By Lemma~\ref{lem-2-4}, a partition of unity argument completes the
proof.
\end{proof}

Lemma~\ref{lem-2-4} contains more information than we needed for the
proof of Theorem~\ref{th-1-5}. However, the full statement will be used to
prove Theorem~\ref{th-1-6} in Section~\ref{sec:3}.

\section{Continuous rational solutions}\label{sec:3}

We will frequently use, not necessarily explicitly referring to it, the
following fact: If $X$ is a~nonsingular real algebraic variety, $X^0
\subset X$ a Zariski open subset, and $U \subset X$ a Euclidean open
subset, then $X^0 \cap U$ is Euclidean dense in $U$.

\begin{lemma}\label{lem-3-1}
Let $X$ be a nonsingular real algebraic variety, $\psi \colon X \to \R$
a regular function, and $f \colon X \setminus Z(\psi) \to \R$ a
continuous rational function. Then there exists an integer $N_0 > 0$
such that for every integer $N \geq N_0$, the function $\psi^Nf$,
extended by $0$ on $Z(\psi)$, is continuous rational on $X$.
\end{lemma}

\begin{proof}
According to a variant of the \L{}ojasiewicz inequality
\cite[Prop.~2.6.4]{bib2}, it suffices to prove that $f$ is a semialgebraic
function. This is straightforward since the graph of the function $f$
restricted to $(X \setminus Z(\psi)) \setminus P(f)$ is a semialgebraic
subset of $(X \setminus Z(\psi)) \times \R$, whose closure is equal to
the graph of $f$.
\end{proof}

\begin{lemma}\label{lem-3-2}
Let $X$ be a nonsingular real algebraic variety and let $\{X^1, \ldots,
X^m\}$ be a Zariski open cover of $X$. Let $f_1, \ldots, f_r, \varphi$
be continuous rational functions on $X$ such that for ${j=1, \ldots, m}$
the restriction $\varphi|_{X^j}$ can be written in the form
\begin{equation*}
\varphi|_{X^j} = \sum_{i=1}^r \varphi_{ij} f_i|_{X^j},
\end{equation*}
where the $\varphi_{ij}$ are continuous rational functions on $X^j$.
Then $\varphi$ can be written in the form
\begin{equation*}
\varphi = \sum_{i=1}^r \varphi_i f_i,
\end{equation*}
where the $\varphi_i$ are continuous rational functions on $X$ with
\begin{equation*}
P(\varphi_i) \subset \bigcup_{j=1}^m (P(\varphi_{ij}) \cup (X \setminus
X^j)).
\end{equation*}
\end{lemma}

\begin{proof}
\setcounter{equation}{0}
We may assume that $X$ is irreducible and the $X^j$ are all nonempty.
Then each $X^j$ is Euclidean dense in $X$. We choose a regular function
$\psi_j$ on $X$ with $Z(\psi_j) = X \setminus X^j$. By
Lemma~\ref{lem-3-1}, there exists a positive integer $N$ such that the
$\varphi_{ij}$ can be written as
\begin{equation}\label{eq-3-2-1}
\varphi_{ij} = \frac{a_{ij}}{\psi_j^N} \quad \textrm{on} \ X^j,
\end{equation}
where the $a_{ij}$ are continuous rational functions on $X$. It follows
that
\begin{equation}\label{eq-3-2-2}
\psi_j^N \varphi = \sum_{i=1}^r a_{ij} f_i
\end{equation}
on $X^j$. By continuity, \eqref{eq-3-2-2} holds on $X$. Multiplying both
sides of \eqref{eq-3-2-2} by $\psi_j^N$ and summing over $j$, we get
\begin{equation}\label{eq-3-2-3}
b\varphi = \sum_{i=1}^r b_i f_i,
\end{equation}
where
\begin{equation*}
b = \sum_{j=1}^m \psi_j^{2N} \quad \textrm{and} \quad b_i = \sum_{j=1}^m a_{ij}
\psi_j^N.
\end{equation*}
The function $b$ is regular with $Z(b) = \varnothing$, which implies
that $\varphi_i \coloneqq b_i/b$ is a continuous rational function on $X$.
In view of \eqref{eq-3-2-3} we have
\begin{equation*}
\varphi = \sum_{i=1}^r \varphi_i f_i.
\end{equation*}
By construction,
\begin{equation*}
P(\varphi_i) \subset \bigcup_{j=1}^m P(a_{ij}),
\end{equation*}
while \eqref{eq-3-2-1} implies
\begin{equation*}
P(a_{ij}) \subset P(\varphi_{ij}) \cup (X \setminus X^j).
\end{equation*}
The proof is complete.
\end{proof}

We will make use of rational maps and rational functions understood in
the standard way. A \emph{rational map} $F \colon X \dashrightarrow Y$,
between real algebraic varieties $X$ and $Y$, is the equivalence class
of regular maps with values in $Y$, defined on Zariski open dense
subsets of $X$; two such regular maps $f_1 \colon X^1 \to Y$ and $f_2
\colon X^2 \to Y$ are declared to be equivalent if $f_1|_{X^0} =
f_2|_{X^0}$ for some Zariski open dense subset $X^0 \subset X^1 \cap
X^2$. We denote by $\dom(F)$ the union of all the domains of regular
maps representing $F$. Thus $F$ determines a regular map $F \colon
\dom(F) \to Y$. The polar set $\pole(F) \coloneqq X \setminus \dom(F)$
is Zariski nowhere dense in $X$. If $Y = \R$, then $F$ is called a
\emph{rational function} on $X$. The rational functions on $X$ form a
ring (a field, if $X$ is irreducible), denoted $\R(X)$.

\begin{definition}\label{def-3-3}
A rational function $R$ on a real algebraic variety $X$ is said to be
\emph{locally bounded} if for every point $p \in X$, one can find a
Zariski open dense subset $X_p \subset X$, a Euclidean open neighborhood
$U_p \subset X$ of $p$, and a real number $M_p > 0$ such that
\begin{equation*}
\abs{R(x)} \leq M_p \quad \textrm{for all} \ x \in U_p \cap \dom(R) \cap
X_p.
\end{equation*}
\end{definition}

It readily follows that the set of all locally bounded rational
functions on $X$ forms a subring of $\R(X)$.

Actually, Definition~\ref{def-3-3} would not be affected if we
substituted for each $X_p$ the set $X^{\mathrm{ns}}$ of all nonsingular
points of X. Definition~\ref{def-3-3} imposes no restriction if the
point $p$ is not in the Euclidean closure of $X^{\mathrm{ns}}$.

\begin{example}\label{ex-3-4}
Consider the Whitney umbrella $W \coloneqq (x^2 = y^2 z) \subset \R^3$.
The set of singular points of $W$ is the $z$-axis. The rational function
$1 / (z+1)$ is locally bounded on $W$, but it is not locally bounded on
the $z$-axis.
\end{example}

We will consider locally bounded rational functions only on nonsingular
varieties. A~typical example is the following.

\begin{example}\label{ex-3-5}
The rational function $xy / (x^2 + y^2)$ on $\R^2$ is locally bounded
(even bounded), but it cannot be extended to a continuous function on
$\R^2$.
\end{example}

\begin{lemma}\label{lem-3-6}
Let $X$ be a nonsingular real algebraic variety and let $R$ be a locally
bounded rational function on $X$. Then the polar set $\pole(R)$ is of
codimension at least $2$.
\end{lemma}

\begin{proof}
Using the inclusion $\R \subset \PB^1(\R)$, we obtain a rational map
$R^* \colon X \dashrightarrow \PB^1(\R)$ determined by $R$. The polar
set $\pole(R^*)$ is of codimension at least $2$ \cite[p.~129,
Thm.~2.17]{bib7}. Since $R$ is locally bounded, we have $\pole(R) =
\pole(R^*)$, which completes the proof.
\end{proof}

\begin{remark}\label{rem-3-7}
Let $X$ be a nonsingular real algebraic variety. Any continuous rational
function $f$ on $X$ determines a rational function $\tilde{f}$ on $X$,
which is represented by the regular function $f|_{X \setminus P(f)}$.
Clearly, $P(f) = \pole(\tilde{f})$. Furthermore, if $g$ is a continuous
rational function on $X$, not identically equal to $0$ on any
irreducible component of $X$, then the quotient $\tilde{f} / \tilde{g}$
is a well defined rational function on $X$ (see Lemma~\ref{lem-2-2}). To
simplify notation, we will prefer to say ``the rational function $f$''
or ``the rational function $f/g$'' instead of writing $\tilde{f}$ or
$\tilde{f}/\tilde{g}$, respectively. For the rational function $f/g$, we
have
\begin{equation*}
\pole(f/g) \subset P(f) \cup P(g) \cup Z(g).
\end{equation*}
\end{remark}

\begin{lemma}\label{lem-3-8}
Let $X$ be an irreducible nonsingular real algebraic variety and let
$\varphi, f_1, \ldots, f_r$ be continuous rational functions on $X$,
where the $f_i$ are not all identically equal to $0$. For $i=1, \ldots,
r$ and $\mathbf{c} = (c_1, \ldots, c_r) \in \R^r$, let
\begin{equation*}
R_{\mathbf{c}i} \coloneqq \frac{(\varphi - (c_1 f_1 + \cdots + c_r f_r))
f_i}{f_1^2 + \cdots + f_r^2}.
\end{equation*}
If $\varphi$ satisfies the PT for $f_1, \ldots, f_r$, then each rational
function $R_{\mathbf{c}i}$ is locally bounded on $X$.
\end{lemma}

\begin{proof} Let $Z \coloneqq Z(f_1, \ldots, f_r)$ and let $S \subset
X$ be an arbitrary subset. The $R_{\mathbf{c}i}$ are well defined
functions on $X \setminus Z$. Setting $R_i = R_{\mathbf{0}i}$, where
$\mathbf{0} = (0, \ldots, 0) \in \R^r$, we get
\begin{equation*}
R_{\mathbf{c}i} = R_i - \frac{c_1 f_i f_1 + \cdots + c_r f_i
f_r}{f_1^2 + \cdots + f_r^2}.
\end{equation*}
Since
\begin{equation*}
\abs{\frac{f_i f_j}{f_1^2 + \cdots + f_r^2}} \leq \frac{1}{2} \quad
\textrm{on} \ X \setminus Z,
\end{equation*}
it follows that $R_{\mathbf{c}i}$ is bounded on $S \cap (X \setminus Z)$
if and only if $R_i$ is such.

Suppose that $\varphi$ satisfies the PT for $f_1, \ldots, f_r$, fix a
point $p \in X$, and set $Z^{(p)} = Z \cup \{p\}$. The function
$\varphi$ can be written in the form
\begin{equation*}
\varphi = \varphi^{(p)} + c_1^{(p)} f_1 + \cdots + c_r^{(p)} f_r,
\end{equation*}
where $c_i^{(p)} \in \R$ and the functions $A_i^{(p)}$ on $X$, defined by
\begin{equation*}
A_i^{(p)} = \frac{\varphi^{(p)}f_i}{f_1^2 + \cdots + f_r^2} \quad
\textrm{on} \ X \setminus Z^{(p)} \quad \textrm{and} \quad A_i^{(p)} = 0
\quad \textrm{on} \ Z^{(p)},
\end{equation*}
are continuous at $p$. In particular, the $A_i^{(p)}$ are bounded on some
Euclidean open neighborhood $U_p \subset X$ of $p$. Consequently, the
functions $R_{\mathbf{c}^{(p)}i}$, where $\mathbf{c}^{(p)} =
(c_1^{(p)}, \ldots, c_r^{(p)})$, are bounded on $U_p \cap (X \setminus
Z)$, which in turn implies that the $R_i$ are bounded on $U_p \cap (X
\setminus Z)$. This conclusion remains valid if $Z$ is replaced by its
Zariski closure $V$ in $X$. The proof is complete since the subset $X
\setminus V \subset X$ is Zariski open dense by Lemma~\ref{lem-2-2}.
\end{proof}

\begin{proof}[Proof of Theorem~\ref{th-1-6}]
It suffices to prove that (\ref{th-1-6-b}) implies (\ref{th-1-6-a})
together with the extra conditions stipulated on the $\varphi_i$. Let us
suppose that (\ref{th-1-6-b}) holds. We may assume that $X$ is
irreducible and the $f_i$ are not all identically equal to $0$.
According to Lemma~\ref{lem-2-4}, for each point $p \in X$, we can find
a Zariski open neighborhood $X^{(p)} \subset X$ of $p$ and continuous
functions $\alpha_1^{(p)}, \ldots, \alpha_r^{(p)}$ on $X^{(p)}$ such
that
\begin{equation*}
\varphi = \alpha_1^{(p)} f_1 + \cdots + \alpha_r^{(p)} \quad \textrm{on} \
X^{(p)}
\end{equation*}
and
\begin{equation*}
\alpha_i^{(p)} = c_i^{(p)} + R_i^{(p)} \quad \textrm{on} \ X^{(p)} \setminus
Z,
\end{equation*}
where $Z = Z(f_1, \ldots, f_r)$, $c_i^{(p)} \in \R$, and
\begin{equation*}
R_i^{(p)} = \frac{(\varphi - (c_1^{(p)} f_1 + \cdots + c_r^{(p)} f_r))
f_i}{f_1^2 + \cdots + f_r^2} \quad \textrm{on} \ X^{(p)} \setminus Z.
\end{equation*}
We regard the $R_i^{(p)}$ as rational functions on $X$ and set
\begin{equation*}
X_0^{(p)} \coloneqq \dom(R_1^{(p)}) \cap \ldots \cap \dom(R_r^{(p)}).
\end{equation*}
Clearly,
\begin{equation*}
X \setminus X_0^{(p)} \subset Z.
\end{equation*}
Furthermore, according to Lemmas~\ref{lem-3-6} and~\ref{lem-3-8}, $X
\setminus X_0^{(p)}$ is a finite set. Since the set $X^{(p)} \setminus
Z$ is Euclidean dense in $X$ (see Lemma~\ref{lem-2-2}), it follows that
\begin{equation*}
\alpha_i^{(p)} = R_i^{(p)} \quad \textrm{on} \ X^{(p)} \cap X_0^{(p)}.
\end{equation*}
Consequently, we obtain a well defined continuous rational function
$\beta_i^{(p)}$ on ${X_1^{(p)} \coloneqq X^{(p)} \cup X_0^{(p)}}$ by
setting
\begin{equation*}
\beta_i^{(p)} = \alpha_i^{(p)} \quad \textrm{on} \ X^{(p)} \quad
\textrm{and} \quad
\beta_i^{(p)} = c_i^{(p)} + R_i^{(p)} \quad \textrm{on} \ X_0^{(p)}.
\end{equation*}
By construction,
\begin{equation*}
\varphi = \beta_1^{(p)} f_1 + \cdots + \beta_r^{(p)} f_r \quad \textrm{on} \
X_1^{(p)}.
\end{equation*}

Now it is easy to complete the proof. We choose a finite collection of
points $p_1, \ldots, p_m$ in~$X$ so that the sets $X^j \coloneqq
X_1^{(p_j)}$ form a cover of $X$. Setting $\varphi_{ij} \coloneqq
\beta_i^{(p_j)}$, we get
\begin{equation*}
\varphi|_{X^j} = \sum_{i=1}^r \varphi_{ij}|_{X^j}, \quad
P(\varphi_{ij}) \subset (P(\varphi) \cup Z \cup P(f_1) \cup \ldots \cup P(f_r))
\cap (X \setminus X_0^{(p_j)}).
\end{equation*}
By Lemma~\ref{lem-3-2}, there exist continuous rational functions
$\varphi_1, \ldots, \varphi_r$ on $X$ such that
\begin{equation*}
\varphi = \varphi_1 f_1 + \cdots + \varphi_r f_r
\end{equation*}
and
\begin{equation*}
P(\varphi_i) \subset \bigcup_{j=1}^r (P(\varphi_{ij}) \cup (X \setminus
X^j)).
\end{equation*}
The functions $\varphi_i$ satisfy all the requirements.
\end{proof}

%\cleardoublepage
\phantomsection

\end{document}